\documentclass{amsart}

\usepackage[OT2,T1]{fontenc}
\usepackage{tilmansdef}
\usepackage{tikz}
\usepackage[colorlinks=true,linkcolor=blue,citecolor=blue]{hyperref} 
\usepackage{tensor}

\usetikzlibrary{matrix,arrows,decorations,cd}

\usepackage{mathtools,stmaryrd}

\DeclareFontFamily{OT1}{pzc}{}
\DeclareFontShape{OT1}{pzc}{m}{it}{<-> s * [1.10] pzcmi7t}{}
\DeclareMathAlphabet{\mathpzc}{OT1}{pzc}{m}{it}

\DeclareMathOperator{\Pol}{Pol}
\DeclareMathOperator{\pol}{pol}

\newcommand{\FSch}[1]{\operatorname{FSch}_{#1}}

\newcommand{\DMod}[1]{\operatorname{Dmod}_{#1}}
\newcommand{\DModV}[1]{\operatorname{Dmod}_{#1}^{V,\operatorname{nil}}}
\newcommand{\DDModF}[1]{\mathbb{D}\operatorname{mod}_{#1}^{F}}
\newcommand{\DDModFc}[1]{\mathbb{D}\operatorname{mod}_{#1}^{F,c}}
\newcommand{\DDModFet}[1]{\mathbb{D}\operatorname{mod}_{#1}^{F,et}}

\newcommand{\DieuFor}{\mathbb D^{\operatorname{f}}}

\newcommand{\FGps}[1]{\operatorname{Fgps}_{#1}} 
\newcommand{\FGpsc}[1]{\operatorname{Fgps}^c_{#1}} 
\newcommand{\AbSch}[1]{\operatorname{AbSch}_{#1}} 
\newcommand{\Hopf}[1]{\operatorname{Hopf}_{#1}}

\newcommand{\Hopfu}[1]{\operatorname{Hopf}^u_{#1}}

\renewcommand{\ell}{\mathpzc{l}}

\DeclareMathOperator{\frob}{frob}

\DeclareMathOperator{\Gal}{Gal}

\DeclareMathOperator{\hull}{hull}

\DeclareMathOperator{\Sep}{Sep}

\newcommand{\modtensor}[2]{\rtimes}
\newcommand{\modcotensor}[2]{\hom}
\newcommand{\Spf}[1]{\operatorname{Spf}\left({#1}\right)}
\newcommand{\Reg}[1]{\mathcal{O}_{#1}}

\DeclareMathOperator{\alg}{alg}
\DeclareMathOperator{\Cof}{Cof}

\DeclareMathOperator{\nil}{nil}

\newtheorem{notation*}[equation]{Notation}

\title{Graded $p$-polar rings and their abelian-group valued functors}
\author{Tilman Bauer}
\date\today

\begin{document}

\begin{abstract}
As an extension of previous ungraded work, we define a graded $p$-polar ring to be an analog of a graded commutative ring where multiplication is only allowed on $p$-tuples (instead of pairs) of elements of equal degree. We show that the free affine $p$-adic group scheme functor, as well as the free formal group functor, defined on $k$-algebras for a perfect field $k$ of characteristic $p$, factors through $p$-polar $k$-algebras. It follows that the same is true for any affine $p$-adic or formal group functor, in particular for the functor of $p$-typical Witt vectors. As an application, we show that the latter is free on the $p$-polar affine line.
\end{abstract}

\maketitle

\section{Introduction} \label{sec:intro}

In \cite{bauer:p-polar}, I introduced the notion of a $p$-polar $k$-algebra, which, roughly speaking, is a $k$-module with a $p$-fold associative and commutative multiplication defined on it. Here $p$ is a prime and $k$ is any commutative ring. If $k$ is a perfect field of characteristic $p$, I showed that the free affine abelian $p$-adic group functor on $\Spec R$ for a $k$-algebra $R$ factors through the category of $p$-polar $k$-algebras, and as a result, so does the functor of points for \emph{every} $p$-adic group defined over $k$.

In this sequel, I prove the corresponding results for graded commutative $k$-algebras, where $k$ is a graded commutative field. Both the definition of a graded $p$-polar $k$-algebra and the proofs are quite distinct, but not independent, from the ungraded case, and the results are more striking in the presence of a grading. This is my excuse for writing a separate paper.

\begin{defn}
Let $k$ be an graded commutative ring with degree-$0$ part $k_0$ and $A$ a graded $k$-module. Let $M_k$ denote the category of $k$-modules $A$ together with a graded symmetric $k_0$-multilinear map $\mu\colon A_j^{\otimes_{k_0} p} \to A_{jp}$, and let $\pol_p\colon \Alg_k \to M_k$ denote the forgetful functor from graded commutative $k$-algebras to $M_k$, where $\mu$ is given by $p$-fold multiplication.

A \emph{graded $p$-polar $k$-algebra} is an object $A \in M_k$ which is a subobject of $\pol_p(B)$ for some algebra $B \in \Alg_k$.
\end{defn}

This definition agrees with the one given in \cite{bauer:p-polar} when both $k$ and $A$ are concentrated in degree $0$ (Lemma~\ref{lemma:olddef}). 

We denote the category of graded $p$-polar $k$-algebra by $\Pol_p(k)$ and its full subcategory of objects that are finite-dimensional as $k$-vector spaces by $\pol_p(k)$.

The restriction functors $\pol_p\colon \Alg_k \to \Pol_p(k)$ and $\pol_p \colon \alg_k \to \pol_p(k)$ defined on graded commutative $k$-algebras (resp. graded commutative $k$-algebras that are finite dimensional as $k$-modules) are called \emph{polarization}. 

\begin{defn}
A \emph{graded field} $k$ is a graded commutative ring over which every graded module is free. This implies that either $k=k_0$ is an ungraded field concentrated in degree $0$ or $k=k_0[u,u^{-1}]$ for an ungraded field $k_0$ and an element $u$ of positive degree $d$, even unless $k_0$ is of characteristic $2$. We say that $k$ is \emph{perfect} if $k_0$ is of characteristic $0$ or if $k_0$ is a perfect field of characteristic $p$ and $p \nmid d$. 
\end{defn}

Our main results parallel those in \cite{bauer:p-polar}. Let $\AbSch{k}$ denote the category of representable, abelian-group-valued functors on $\Alg_k$, and let $\AbSch{k}^p$ denote the full subcategory of functors taking values in abelian pro-$p$-groups. Let $\FGps{k}$ be the category of ind-representable functors on $\alg_k$ taking values in abelian groups. We will refer to objects of $\AbSch{k}$ and $\FGps{k}$ as affine and formal groups, respectively.

\begin{thm}\label{thm:freeschemefactorization}
Let $k$ be a perfect field of characteristic $p$ and denote by $\Fr(R)$ the left adjoint of the forgetful functors $\AbSch{k}^p \to \Alg_k^{\op}$ resp. $\FGps{k} \to \alg_k^{\op}$. Then $\Fr$ factors through $\pol$:
\[
\begin{tikzcd}
\Alg_k^{\op} \arrow[r,"\Fr"] \ar[dr,swap,"\pol"] & \AbSch{k}^p;\\
& \Pol_p(k)^{\op} \ar[u,"\tilde \Fr"]\\
\end{tikzcd}
\qquad
\begin{tikzcd}
\alg_k^{\op} \arrow[r,"\Fr"] \ar[dr,swap,"\pol"] & \FGps{k}\\
& \pol_p(k). \ar[u,"\tilde \Fr"]\\
\end{tikzcd}
\]
\end{thm}

\begin{corollary}\label{thm:affineschemefactorization}
Let $k$ be a perfect field of characteristic $p$ and $M \in \AbSch{k}^p$ or $M \in \FGps{k}$. Then $M$ factors uniquely through $\pol$:
\[
\begin{tikzcd}
\Alg_k \arrow[r,"M"] \ar[dr,swap,"\pol"] & \{\text{abelian pro-$p$-groups}\};\\
& \Pol_p(k) \ar[u,"\tilde M"]\\
\end{tikzcd}
\qquad
\begin{tikzcd}
\alg_k \arrow[r,"M"] \ar[dr,swap,"\pol"] & \Ab\\
& \pol_p(k). \ar[u,"\tilde M"]\\
\end{tikzcd}
\]
\end{corollary}
\begin{proof}
Given any $M \in \AbSch{k}$ or $M \in \FGps{k}$, we have that
\[
M(R) = \Hom(\Spec R, M) = \Hom(\Fr(R),M),
\]
where the last Hom group is of objects of $\AbSch{k}^p$ or $\FGps{k}$, respectively. The Corollary follows from Thm.~\ref{thm:freeschemefactorization}.
\end{proof}

To prove Thm.~\ref{thm:freeschemefactorization}, we define graded $p$-typical Witt vectors $W(R)$ and co-Witt vectors $CW(R)$ for \emph{evenly graded} $p$-polar $k$-algebras $R$ and Dieudonné equivalences
\[
D\colon \AbSch{k}^p \to \DMod{k}^p
\]
and
\[
\DieuFor\colon \FGps{k}^p \to \DDModF{k}
\]
with certain categories of $W(k)$-modules with Frobenius and Verschiebung operations. Here $\FGps{k}^p$ denotes the full subcategory of $\FGps{k}$ of functors taking values in abelian $p$-groups.

We prove:
\begin{thm} \label{thm:dieudonneoffree}
Let $R \in \Alg_k$ be an evenly graded commutative $k$-algebra when $p>2$, or a 
graded, commutative $k$-algebra when $p=2$.

There are natural isomorphisms
\[
\DieuFor(\Fr(R)) \cong CW(R) \quad \text{for finite-dimensional $R$}
\]
and
\[
D(\Fr(R)) \cong CW^u(R) \oplus \left(\mu_{p^\infty}(R\otimes_k \bar k) \otimes W(\bar k)\right)^{\Gal(k)} \quad \text{for $R \in \Alg_k$,}
\]
where in the last factor, $\mu_{p^\infty}(A)$ denotes the $p$-power torsion in $A^\times$, invariants of the absolute Galois group $\Gal(k)$ acting diagonally on $\bar k$ and $W(\bar k)$ are taken, and $CW^u$ denotes the unipotent part of $CW$.
\end{thm}

Because the right-hand side is defined on $p$-polar algebras, so is the left hand side and hence, since $D$ is an equivalence, also $\Fr(R)$. This is the main ingredient in proving Thm.~\ref{thm:freeschemefactorization}. 

As an application, we show:
\begin{thm}\label{thm:lambdapcofree}
The affine group scheme of $p$-typical Witt vectors is the free unipotent abelian group scheme on the $p$-polar affine line, i.e. on the free $p$-polar algebra on a single generator. 
\end{thm}

This is an analog of the fact that the Hopf algebra representing the big Witt vectors, the algebra $\Lambda$ of symmetric functions, is cofree on a polynomial ring in one generator \cite{hazewinkel:cofree-coalgebras}. The corresponding Hopf algebra $\Lambda_p$ for $p$-typical Witt vectors is definitely not cofree on a (non-polar) algebra.

\section{\texorpdfstring{Graded $p$-polar algebras}{Graded p-polar algebras}}

We begin the study of $p$-polar $k$-algebras with some observations and examples.

\begin{remark}
A graded $p$-polar $k$-algebra $A$ does not supply a map $A^{\otimes_k p} \to A$ -- only elements of the same degree can be multiplied together. In particular, a graded $p$-polar $k$-algebra is not a $p$-polar $k$-algebra when one forgets the grading.
\end{remark}

\begin{remark}
The embeddability $i\colon A \to B$ into a graded commutative algebra can be thought of as saying that for any elements $x_1,\dots,x_n \in A$ and scalars $\lambda_1,\dots,\lambda_m \in k$, there is at most one way of multiplying them together using $\mu$ (up to sign); namely, the element $\lambda_1\cdots \lambda_k i(x_1) \cdots i(x_n) \in B$, which is either in the image of $i$ or it isn't. 
\end{remark}

\begin{example}
If $k=k_0$, the submodule $k\langle x^{p^i} \mid i \geq 0\rangle \subset k[x]$ is a sub-$p$-polar algebra of $\pol(k[x])$, where $|x|>0$. It is the free $p$-polar algebra on a generator $x$. This shows that in contrast to the ungraded case, even for $p=2$, $p$-polar algebras are much weaker structure than actual algebras.
\end{example}

\begin{remark}\label{remark:ptypicalsplitting}
If $k=k_0$, it is apparent that for any $p$-polar algebra $A$, the submodule $A_{(j)} = \bigoplus_{i \geq 0} A_{jp^i}$ is a polar subalgebra and direct factor, and that
\[
A \cong A_0 \times \prod_{p \nmid j} A_{(j)}.
\]

If $k = k_0[u,u^{-1}]$ with $|u| = d>0$ and $p \nmid d$, we see instead that
\[
A_{(j)} = \sum_{i \geq 0, l \in \Z} A_{jp^i+dl}
\]
is a polar subalgebra and direct factor, and that
\[
A \cong \prod_j A_{(j)},
\]
where $j$ runs through the residue classes of $\{jp^i \mid i \geq 0\}$ in $\Z/d\Z$.

We call a $p$-polar algebra of this form $A_{(j)}$ a \emph{$p$-typical} polar algebra of degree~$j$. The inclusion of such $p$-polar algebras into all $p$-polar algebras is biadjoint to the functor $A \mapsto A_{(j)}$. We write $\pol_{(j)}(A) := (\pol(A))_{(j)}$. This is a sub-$p$-polar algebra, but not a subalgebra of $A$.

In particular, if $p>2$, we see that every graded $p$-polar $k$-algebra splits as a product $A=A_{\odd} \times A_{\ev}$, where $A_{\odd} = \bigoplus_n A_{2n+1}$ and $A_{\ev}$ the obvious complement.
\end{remark}

\begin{example}
Consider the stable splitting $P=\Sigma^\infty (\CP^\infty)\hat{{}_p} \simeq P_1 \vee \cdots \vee P_{p-1}$ of the $p$-completion of complex projective space \cite{mcgibbon:rank-1-loops} with
\[
H^*(P_i) = \langle x^j \mid j \equiv i \pmod{p-1}\rangle < \F_p[x] = H^*(\CP^\infty,\F_p).
\]
By \cite{sullivan:genetics}, $P_{p-1}$ is the suspension spectrum of a space (the classifying space of the Sullivan sphere), but all other $P_i$ are not.
However, the maps
\[
P_i \hookrightarrow P = \Sigma^\infty (\CP^\infty)\hat{{}_p} \xrightarrow{\Sigma^\infty \Delta} \Sigma^\infty ((\CP^\infty)\hat{{}_p})^p \simeq P \wedge \cdots \wedge P \twoheadrightarrow P_i \wedge \cdots \wedge P_i
\]
induces a $p$-polar algebra structure on $\tilde H^*(P_i)$, and the splitting $P\simeq P_1 \vee \cdots \vee P_{p-1}$ induces a splitting of $p$-polar algebras in cohomology. In fact,
\[
H^*(P_i;\F_p) \cong \bigoplus_{\substack{j \equiv i \;\;(p-1)\\p \nmid j}} \pol_{(j)} H^*(P;\F_p)
\]
So while the $P_i$ are not spaces for $i \neq p-1$, they do retain some likeness to spaces in that their cohomologies are $p$-polar algebras. This raises the question whether there is a reasonable notion of a ``$p$-polar space'' somewhere between connective spectra and spaces.
\end{example}

In a way, the definition of a $p$-polar $k$-algebra is wrong in the same way the definition of a manifold as a submanifold of $\mathbf R^n$ is wrong; it mentions an enveloping object which is not part of the data. The following proposition remedies this to a certain extent:

\begin{prop}
The functor $\pol\colon \Alg_k \to M_k$ has a left adjoint given for $A \in M_k$ by
\[
A \mapsto \hull(A) = \Sym(A)/(x_1\cdots x_p - \mu(x_1,\dots,x_p) \mid x_1,\dots,x_p \in A_i).
\]
An object $A \in M_k$ is a $p$-polar $k$-algebra iff the unit map of this adjunction, $u\colon A \to \pol(\hull(A))$, is injective.
\end{prop}
\begin{proof}
The existence and structure of the left adjoint, $\hull$, is obvious.

If $u$ is injective, $A$ is a $p$-polar algebra by definition. Conversely, if $A$ is $p$-polar, say $i\colon A \hookrightarrow \pol(B)$ for some $B \in \Alg_k$, then by the universal property of the left adjoint, there is a factorization
\[
\begin{tikzcd}
A \arrow[dr,hook,"i",swap] \arrow[r,"u"] & \pol(\hull(A)) \arrow[d,dashed]\\
& \pol(B).
\end{tikzcd}
\]
Since $i$ is injective, so is $u$.
\end{proof}

It is possible to give a list of axioms for objects of $M_k$ to be a $p$-polar algebra, but that list becomes quite unwieldy in the general case. We will only do this in important special cases.

\begin{lemma}\label{lemma:olddef}
If $k=k_0$ and $A=A_0$ (i.e. in the ungraded case), the definition of a $p$-polar $k$-algebra agrees with the one given in \cite{bauer:p-polar}; i.e., $A \in M_k$ is $p$-polar iff 
\begin{description}
	\item[(ASSOC)]  \label{lemma:olddef:homogeneousassoc}
For the symmetric group $\Sigma_{2p-1}$ permuting the elements $x_1,\dots,x_p$, $y_2,\dots,y_p \in A$,
\[
\mu(\mu(x_1,\dots,x_p),y_2,\dots,y_p)
\]
is $\Sigma_{2p-1}$-invariant.
\end{description}
\end{lemma}
\begin{proof}
Clearly, axiom (ASSOC) holds if $A$ is $p$-polar. Conversely, suppose that (ASSOC) holds, and let $i\colon A \to \hull(A)$ be the adjunction unit. In \cite{bauer:p-polar}, it was shown that (ASSOC) implies that for any $i \geq 0$ and any set of $1+i(p-1)$ elements $x_1,\dots,x_{1+i(p-1)}$, there is exactly one way of multiplying the $x_i$ together using $\mu$, and any other number of elements cannot be multiplied together. Write $\mu(x_1,\dots,x_{1+i(p-1)})$ for this unique product. Let $j\colon \bigoplus_{i=0}^\infty \Sym^{1+i(p-1)}(A) \hookrightarrow \Sym(A)$ be the inclusion and
\begin{align*}
\hull(A) \supseteq B =& \left(\bigoplus_{i=0}^\infty \Sym^{1+i(p-1)}(A)/j^{-1}((x_1\cdots x_p - \mu(x_1,\dots,x_p)))\right)\\
\cong & \left(\bigoplus_{i=0}^\infty \Sym^{1+i(p-1)}(A)\right)/\langle x_1\cdots x_{1+i(p-1)} - \mu(x_1,\dots,x_{1+i(p-1)})\rangle\\
\cong & A
\end{align*}
We have thus exhibited $A$ as a subobject of $\hull(A)$.
\end{proof}

Next, we consider the important case of a graded $p$-polar $k$-algebra over an ungraded ring $k=k_0$.

\begin{lemma}\label{lemma:nonperiodicchar}
Let $A$ be an object in $M_k$, where $k=k_0$ is ungraded. Then $A$ is a $p$-polar $k$-algebra iff
\begin{enumerate}
	\item \label{lemma:nonperiodicchar:zero} $A_0$ is a $p$-polar $k$-algebra, and
	\item \label{lemma:nonperiodicchar:assoc} For the symmetric group $\Sigma_{2p}$ permuting the elements $x_1,\dots,x_{2p} \in A_j$ and elements $y_3,\dots,y_p \in A_{pj}$,
	\[
	\mu(\mu(x_1,\dots,x_p),\mu(x_{p+1},\dots,x_{2p}),y_3,\dots,y_p)
	\]
is $\Sigma_{2p}$-invariant (up to multiplication with the sign of the permutation if $j$ is odd).
\end{enumerate}
\end{lemma}
\begin{proof}
Again, the implication $A$ $p$-polar $\Rightarrow$ \eqref{lemma:nonperiodicchar:zero}, \eqref{lemma:nonperiodicchar:assoc} is straightforward. For the converse, we may assume by Remark~\ref{remark:ptypicalsplitting} without loss of generality that either $A=A_0$ or $A = A_{(j)}$ is $p$-typical. Lemma~\ref{lemma:olddef} takes care of the first case, so assume $A$ is $p$-typical and \eqref{lemma:nonperiodicchar:assoc} holds.

For any graded $p$-typical $k$-module $M = M_{(j)}$, the free object $T_M$ in $M_k$ on $M$ is given inductively by
\[
(T_M)_n = M_n \oplus \Sym^p((T_M)_{\frac n p}),
\]
where $(T_M)_n = 0$ if $n \not\in \Z$. We call an element of $(T_M)_n$ a monomial if it is either an element of $M_n$ or a monomial $\{x_1,\dots,x_p\} \in \Sym^p$ on monomial elements in $x_i \in (T_M)_{\frac n p}$, using curly braces for equivalence classes of tensors in $\Sym^p$. Clearly, by linearity, any element of $T_M$ is a linear combination of monomials. One could describe these monomial elements as some kinds of labelled trees. While this is a good picture to have in mind, I will not use that language.

Define an equivalence relation $\sim$ on monomials in $T_M$ (and hence, by linear extension, on all of $T_M$) generated by 
\begin{multline*}
\{\{x_1,\dots,x_p\},\{x_{p+1},\cdots,x_{2p}\},y_3,\cdots, y_p\}\\
 \sim (-1)^{|x_1|} \{\{x'_1,\dots,x'_{p-1},x'_{p+1}\},\{x'_{p},x'_{p+2},\cdots,x'_{2p}\},y'_3,\cdots, y'_p\}
\end{multline*}
iff $y_j \sim y'_j$ for $3 \leq h \leq p$, $x_i \sim x'_i$ for $1 \leq i \leq 2p$.
Then $T_M/\sim$ is the free object in $M_k$ satisfying \eqref{lemma:nonperiodicchar:assoc}. (Obviously, the $\Sigma_{2p}$-equivariance is equivalent to the equivariance under interchanging $x_p$ and $x_{p+1}$, given the guaranteed $\Sigma_p \times \Sigma_p$-equivariance.)

There is a linear map $f\colon T_M \to \Sym(M)_{(j)}$ given on monomials by $f(m) = m$ for $m \in M$ and $f(\{x_1,\dots,x_p\}) = \{f(x_1),\dots, f(x_p)\}$ for $\{x_1,\dots, x_p\} \in \Sym^p(T_M)$. We claim that this map induces an injective map on $T_M/\sim$ with image $\Sym(A)_{(j)}$.

To see this is the image, let $X=\{x_1,\dots,x_n\} \in \Sym(M)_{jp^N}$. If $n=1$, $\{x_1\} = f(x_1)$ and we are done. Otherwise, because $M$ is $p$-typical, there has to be a partition of $\{1,\dots,n\}$ into $p$ parts $I_1,\dots,I_p$ such that for $X_i = \pm \{x_j \mid j \in I_i\}$, $|X_i| = jp^{N-1}$. Inductively, all $X_i$ are in the image of $f$, hence so is $X=f({X_1,\dots,X_p})$.

We proceed to show injectivity.

Let $x \in T_M$ be a monomial. We say that $y \in T_M$ occurs at depth $d$ in $x$ if either $d=0$ and $y=x$ or $x=\{x_1,\dots,x_p\}$ and $y$ occurs at depth $d-1$ in $x_i$ for some $i$.

Now suppose that $y_1$ and $y_2$ occur at a common depth $d \geq 1$ in $x=\{x_1,\dots,x_p\}$, and let $x' \in T_M$ be the element obtained by interchanging $y_1$ and $y_2$. Then I claim that $x \sim \pm x'$. To see this, we proceed by induction. If $d=1$ then the claim is true by symmetry. Suppose that $d>1$. Then $y_1$ occurs at depth $d-1$ in some $x_i$ and $y_2$ occurs at depth $d-1$ in some $x_j$. If $i=j$, we are done by induction. Otherwise, suppose without loss of generality that $i=1$ and $j=2$. Let $x_1=\{x_{11},\dots,x_{1p}\}$ and $x_2=\{x_{21},\dots,x_{2p}\}$. Without loss of generality, suppose that $y_i$ occurs at depth $d-2$ in $x_{i1}$ for $i=1,2$. Then
\[
x \sim \pm x^{(1)}=\{x^{(1)}_1,\dots,x^{(1)}_p\} = \{\{x_{11},x_{21},x_{13},\dots,x_{1p}\},\{x_{12},x_{22},\dots,x_{2p}\},x_3,\dots,x_p\}.
\]
Then $y_1$ and $y_2$ occur at depth $d-1$ in $x^{(1)}_1$ and by induction, $x^{(1)}_1 \sim (x^{(1)})_1' = \{x_{11}',x_{21}',x_{13},\dots,x_{1p}\}$, the element obtained from $x^{(1)}_1$ by interchanging $y_1$ and $y_2$. But then
\begin{multline*}
x \sim \pm \{\{x_{11}',x_{21}',x_{13},\dots,x_{1p}\},\{x_{12},x_{22},\dots,x_{2p}\},x_3,\dots,x_p\} \\
\sim \{\{x_{11}',x_{12},\dots,x_{1p}\},\{x_{21}',x_{22},\dots,x_{2p}\},x_3,\dots,x_p\} = x'.
\end{multline*}
We conclude that if $x$, $x' \in T_M$ with $f(x)=f(x')$ (i.e. they contain the same leaf elements at any given level), then $x \sim x'$.

Now let $A=A_{(j)}$ be an object of $M_k$ such that \eqref{lemma:nonperiodicchar:assoc} holds and $F_A$ the quotient of $(T_A/\sim) \cong \Sym(A)_{(j)}$ by the intersection of the ideal $(x_1\dots x_p-\mu(x_1,\dots,x_p)) \triangleleft \hull(A)$ with $\Sym(A)_{(j)}$. Then the map $A \mapsto F_A$ is an isomorphism, showing that $A \hookrightarrow \hull(A)$.
\end{proof}

We finish the section with a characterization of the image of $A \to \hull(A)$, i.e. a determination of which elements can be multiplied together.

\begin{lemma}\label{lemma:stupidmorita}
For $p \nmid d$, denote by $k(d)$ the field $k=k_0[u^{\pm 1}]$, where $|u|=d$ and $k_0$ is a fixed field of characteristic $p$. For $p \nmid j$, let $h(j,d)$ be the smallest $h$ such that $d \mid (p^h-1)j$. Denote by $\Mod_{(j)}(d)$ the category of graded $k(d)$-modules $M$ which are $p$-typical in the sense that $M = \bigoplus_{i=0}^{h(j,d)} M_{jp^i}$.
Then the abelian categories $\Mod_{(j)}(d)$ and $\Mod_{(j')}(d')$ are equivalent iff
$h(j,d)=h(j',d')$.
\end{lemma}
\begin{proof}
Suppose $h(j,d)=h(j',d')$. For $M \in \Mod_{(j)}(d)$,
\[
M = \prod_{i=0}^{h(j,d)} M_{jp^i}.
\]
Hence the desired equivalence of categories is given by a regrading
\[
\Pol_{(j)}(d) \to \Pol_{(j')}(d'), \quad M \to M',
\]
where
\[
(M')_{j'p^i} = M_{jp^i}. 
\]

Conversely, observe that any projective generator of $\Mod_{(j)}(d)$ must have $k_0$-dimension at least $h(j,d)$. Since any equivalence $\Psi\colon \Mod_{(j)}(d) \simeq \Mod_{(j')}(d')$ sends projective generators to projective generators, it follows that $h(j,d)=h(j',d')$.
\end{proof}

\begin{corollary}\label{cor:equivalentalgebracategories}
Denote by $\Pol_{(j)}(d)$ the category of $p$-typical, $p$-polar $k(d)$-algebras of degree $j$. Then $\Pol_{(j)}(d) \simeq \Pol_{(j')}(d')$ iff $h(j,d)=h(j',d')$.
\end{corollary}
\begin{proof}
One direction follows directly from Lemma~\ref{lemma:stupidmorita} because the equivalence of categories $\Mod_{(j)}(d) \to \Mod_{(j')}(d')$ extends to the desired equivalence of $p$-polar algebras. Counting the minimal dimension of a generator for $\Pol_{(j)}(d)$ gives the reverse conclusion.
\end{proof}

\begin{corollary}\label{cor:primefields}
For any $p \nmid d$ and $d \nmid j$, there exists $h \geq 0$ such that
\[
\Pol_{(j)}(d) \simeq \Pol_{(1)}(p^h-1).
\]
\end{corollary}
\begin{proof}
$h(1,p^h-1) = h$.
\end{proof}

\begin{prop}\label{prop:imageofiotacycliccase}
Let $A=A_{(j)}$ be a $\Z/d\Z$-graded, $p$-typical polar $k$-algebra, where $p \nmid d$. Denote by $h$ the smallest positive integer satisfying $p^hj \equiv j \pmod{d}$. Let $x_1,\dots,x_n$ be elements in $A$, of which exactly $n_i$ have degree $jp^i$ for $i=0,\dots,h-1$, and hence $\sum n_i = n$. Then $x_1\cdots x_n \in A_{jp^m}$ is in $\im(\iota\colon A \to \hull(A))$ iff
\[
\sum_{\alpha=0}^{h-1} n_\alpha p^\alpha \equiv p^m \pmod{p^h-1}.
\]
\end{prop}
In this proposition, the $\Z$-graded case is included with $d=0$ and $h = \infty$.

\begin{proof}
Denote by $S$ the set of sequences $\underline n = (n_0,\dots,n_{h-1})$ of nonnegative integers, and call $\underline n$ \emph{multipliable} if for some, and hence any, set $x_1,\dots,x_n$ of elements of $A$ with degrees prescribed by $\underline n$, $x_1\cdots x_n \in \im(\iota)$. Define an equivalence relation on $S$ by $\underline n \sim \underline n'$ iff either both or none of $\underline n$ and $\underline n'$ are multipliable.

We have that
\[
(n_0,\dots,n_i+p,\dots,n_{h-1}) \sim (n_0,\dots,n_i,n_{i+1}+1,\dots,n_{h-1})
\]
for $0 \leq i < h-1$ because if in the sequence $x_1,\dots,x_n$, the elements $x_1,\dots,x_p$ have degree $jp^i$, $x_1\dots x_n \in \im(\iota)$ iff $\mu(x_1,\dots,x_p) x_{p+1}\cdots x_n \in \im(\iota)$. 
Thus, inductively,
\[
(n_0,\dots,n_{h-1}) \sim (n_0+n_1p+\cdots+n_{h-1}p^{h-1},0,\dots,0).
\]
The question is thus reduced to the question of when for $x_1,\dots,x_n \in A_j$, $x_1\dots x_n \in \hull(A)_{jp^m}$ is in the image of $\iota$. If $k \neq p^m \pmod{p^j-1}$ for any $k$, $A_{jn} = 0$ and hence this is a necessary condition. If $k=p^m$ for some $m$, $x_1\cdots x_n \in \im(\iota)$. Furthermore, since $p^hj \equiv j \pmod {d}$, $\mu(x_1,\dots,x_{p^h}) \in A_j$, so the product of any $n=p^m + l (p^h-1)$ elements of degree $j$ is in $\im(\iota)$. 
\end{proof}

\begin{corollary}\label{cor:retract}
Let $k$ be a perfect graded field with $|u|=d$ and $A=A_{(j)}$ be a $p$-typical polar $k$-algebra with $d \nmid j$. Let $\phi$ be the regrading isomorphism
\[
\phi\colon \Pol_{(j)}(d) \to \Pol_{(1)}(p^{h(j,d)}-1).
\]
Then $A \to \phi^{-1}\left(\hull(\phi(A))_{(1)}\right)$ is an isomorphism.
\end{corollary}
\begin{proof}
When $j=1$ and $d=p^h-1$, $\phi$ is the identity and the claim follows from Prop.~\ref{prop:imageofiotacycliccase}. The general case follows because $\phi$ is an equivalence.
\end{proof}


\section{Graded Witt vectors}

In this section, we will consider commutative graded rings $A$ instead of graded-commutative rings, i.e. graded rings $A$ whose underlying ungraded ring is commutative. We will apply the results of this section to evenly graded, graded-commutative rings or graded, commutative rings over fields of characteristic $2$. There does not seem to be an adequate (for our purposes) definition of graded-commutative Witt vectors for graded-commutative rings, nor will it be necessary, in light of Prop.~\ref{prop:oddfree} below.

Throughout, let $p$ be a fixed prime. For a graded abelian group $M$ and an integer $i \geq 0$, we write $M(i)$ for the graded abelian group with $M(i)_n = M_{p^in}$.

We assume the reader is familiar with the ungraded theory of $p$-typical Witt vectors, cf.~\cite{witt:witt-vectors,hazewinkel:witt,hesselholt:witt-survey}.
\begin{defn}
Let $A$ be a commutative graded ring. As a graded set, the $p$-typical Witt vectors of $A$ of length $0 \leq n \leq \infty$ are defined as
\[
W_n(A) =\prod_{i=0}^n A(i), \quad \text{i.e.} \quad W(A)_j = \prod_{i=0}^n A_{jp^i}.
\]
Just as in the classical case, there is a ghost map
\[
w\colon W_n(A) \to \prod_{i=0}^n A(i)
\]
given by 
\[
w(a_0,a_1,\dots) = (a_0,a_0^p+pa_1,a_0^{p^2}+pa_1^p+p^2a_2,\dots).
\]
There is a unique functorial ring structure on $W_n(A)$ making $w$ a homomorphism of graded rings.
\end{defn}

If $A^u$ denotes $A$ as an ungraded ring then $W(A^u)$ and $W(A)^u$ are in general distinct:
\begin{example}
\begin{itemize}
	\item $W(k_0[u]) \cong W(k_0)[u]$ if $k_0$ is a ring concentrated in degree $0$ and $|u| =d>0$. This is false if $d=0$: the (ungraded) $p$-typical Witt vectors of $W(\F_p[x])$ are more complicated (cf. \cite[Exercise~10]{borger:witt-vector-lectures}). 

	\item $W(k_0[u^{\pm 1}]) \cong W(k_0)[u^{\pm 1}]$ in the same situation.
\end{itemize}
\end{example}

The Verschiebung $V\colon W(A)(1) \to W(A)$ is the map given by $V(x_0,x_1,\dots) = (0,x_0,x_1,\dots)$; the Teichm\"uller map is the multiplicative map $A \to W(A)$, $x \mapsto \underline x = (x,0,0,\dots)$. Furthermore, the Frobenius map is characterized as the unique natural map $F\colon W(A) \to W(A)(1)$ with the property that if $w(x) = (w_0,w_1,\dots) \in \prod_i A(i)$ then
\[
w(F(x)) = (w_1,w_2,\dots) \in \prod_i A(i+1)
\]
The maps $V$ and $F$ restrict to maps $V\colon W_n(A)(1) \to W_{n+1}(A)$ and $F\colon W_n(A) \to W_{n-1}(A)(1)$, and we have that $F \circ V = p$ and $F(\underline a) = \underline{a^p}$.

Now let $k = k_0[u^{\pm 1}]$ be a graded field. Given a graded $W(k)$-module $M$ and $j\geq 0$, the abelian group $M(j)$ obtains a $W(k)$-linear structure $\alpha\cdot m = \frob(\alpha) m$, where $\frob$ is the Frobenius map on $W(k)$, the unique lift of the $p$th power map on $k$.

\begin{lemma}\label{lemma:shiftperiodicity}
If $k=k_0[u^{\pm 1}]$ is perfect with $|u|=d>0$ then the functor $(1)\colon M \mapsto M(1)$ is an equivalence, and $(l)$ is a naturally isomorphic with the identity for some $l \geq 1$.
\end{lemma}
\begin{proof}
Since $M(j_1)(j_2)=M(j_1+j_2)$, it suffices to show the second claim. Since $k$ is perfect, $p$ is a unit in $\Z/d\Z$ and hence there is a smallest $l \geq 1$ such that $p^l \equiv 1 \pmod{d}$, say $p^l-1 = nd$. Then the map
\[
M_j \to M(l)_j = M_{jp^l}; \quad m \to u^{nj}m
\]
is a $W(k)$-linear isomorphism. 
\end{proof}

Note that $M$ in fact obtains a natural $\Z[\frac 1p]$-grading by setting $M_{\frac n {p^k}} = M(-k)_n$.

If $k=k_0$, the functor $(1)$ is not an equivalence, but it has a right inverse $(-1)$ given by
\[
M(-1)_n = \begin{cases} 0;& p \nmid n\\
M_{\frac n p}; & p \mid n\end{cases}
\]
with the $W(k)$-linear structure given by $\alpha.m = \frob^{-1}(\alpha)m$. The Frobenius $\frob$ is invertible because $k=k_0$ is perfect. Confusingly, $(-1)$ being a right inverse means that $M(-1)(1)\cong M$.

In \cite[\textsection II.1.5]{fontaine:groupes-divisibles}, the group of co-Witt vectors $CW(A)$ is defined for an ungraded ring $A$, containing the subgroup of unipotent co-Witt vectors $CW^u(A) = \colim(W_0(A) \xrightarrow{V} W_1(A) \xrightarrow{V} \cdots)$. As a set,
\[
CW(A) = \left\{(a_i) \in \prod_{i \leq 0} A \mid (\dots,a_{-r-1},a_{-r}) \text{ is a nilpotent ideal for some $r\geq 0$}\right\},
\]
and $CW^u(A)$ consists of those $(a_i)$ with almost all $a_i=0$.

\begin{prop}\label{prop:gradedCW}
For a perfect graded field $k$ and a commutative, graded $k$-algebra $A$, define the set
\[
CW(A)_j = \Bigl(\prod_{i \leq 0} A(i) \Bigr)_j \cap CW(A^u),
\]
where, as before, $X^u$ denotes the object $X$ with the grading forgotten.
Then $CW(A)$ is a $W(k)$-module such that $CW(A)^u$ is a subgroup of $CW(A^u)$. It is stable under the Frobenius and Verschiebung operators, and contains as a submodule
\[
CW^u(A) = \colim(W_0(A) \xrightarrow{V} W_1(A(-1)) \xrightarrow{V} W_2(A(-2)) \xrightarrow{V} \cdots),
\]
where $V\colon W_n(A) = W_n(A(-1))(1) \to W_{n+1}(A(-1))$ is induced by the Verschiebung $V\colon W_n(A)(1) \to W_{n+1}(A)$.
\end{prop}
Note that $W_n(A)(-1) \hookrightarrow W_n(A(-1))$ is an isomorphism iff $k=k_0[u^{\pm 1}]$ or $A = A_0$.
\begin{proof}
We need to show that the addition in $CW(A^u)$ preserves the grading. If $S_m \in \Z[x_0,\dots,x_m,y_0,\dots,y_m]$ denotes the addition polynomial in $W_n(A)$, i.e. such that
\[
\Bigl((a_0,\dots,a_m)+(b_0,\dots,b_m)\Bigr)_m = S_m(a_0,\dots,a_m,b_0,\dots,b_m),
\]
then the addition on $CW(A^u)$ is defined in such a way that if $(\dots,a_{-1},a_0)+(\dots,b_{-1},b_0) = (\dots,c_{-1},c_0)$ then 
\[
c_{-n} = S_m(a_{-m-n},\dots,a_{-n},b_{-m-n},\dots,b_{-m})
\]
for $m \gg 0$, and it is shown in \cite[\textsection II.1.5]{fontaine:groupes-divisibles} that this gives a well-defined group structure. Since the polynomials $S_m$ are homogeneous when the variables $x_i$ and $y_i$ are given degree $jp^i$, the result follows.
\end{proof}

\subsection{Representability of Witt and co-Witt vectors}

Since $W(A)_j \cong \prod_{i=0}^\infty A_{jp^i}$ as sets, this set-valued functors is represented by
\[
(\Lambda_p)_j=k[\theta_{j,0},\theta_{j,1},\dots],
\]
where $|\theta_{j,i}| = jp^i$, and $W(A)$, as a graded object, is represented by the bigraded $k$-algebra $\Lambda_p = (\Lambda_p)_*$. Each $(\Lambda_p)_j$ obtains a Hopf algebra structure by the natural Witt vector addition on $W(A)_j$, and $\Lambda_p$ becomes a Hopf ring (cf. \cite{wilson:hopf-rings}) with a comultiplication
\[
\Lambda_p(A)_j \to \bigoplus_{j_1+j_2=j} \Lambda_p(A)_{j_1} \otimes \Lambda_p(A)_{j_2}.
\]
In other words, $\Lambda_p$ represents a graded ring object, even a plethory, in affine schemes. This is a graded version of the $p$-typical symmetric functions of \cite[II.13]{borger-wieland:plethystic}.

The co-Witt vectors are not representable, but their restriction to $\alg_k$, i.e. finite-dimensional $k$-algebras, is ind-representable, that is, $CW_k$ is a formal group. In the ungraded case, this is described in \cite[\textsection II.3--4]{fontaine:groupes-divisibles}. In our graded case, $CW(A)_j$ is represented by the profinite ring
\[
(\Reg{CW_k})_j = 
\lim_{m,n \geq 0} k[x_{j,0},x_{j,-1},\dots]/\Bigl(x_{j,-n},x_{j,-n-1},\dots,)^m
\]
with $|x_{j,i}| = jp^i$, which in the case $k=k_0$ and $jp^i \not\in\Z$ is to be understood as $x_{j,i}=0$. By naturality of the co-Witt vector addition, $CW_k$ thus becomes a (graded) formal group.


\subsection{Witt vectors of \texorpdfstring{$p$}{p}-polar rings}

Observe that the definition of the abelian group of graded Witt vectors makes sense if $A$ merely is a graded $p$-polar ring. Moreover, if $A$ is a $p$-polar graded $k$-algebra, for a commutative graded ring $k$, then $W(A)$ is a $W(k)$-module and in fact a $p$-polar graded $W(k)$-algebra.

\begin{lemma} \label{lemma:wittonptypical}
If $k$ is a perfect graded field then the Witt vector functor restricts to
\[
W_n\colon \Pol_{(j)}(k) \to \Mod_{(j)}(W(k)),
\]
and the Frobenius and Verschiebung operators restrict to $\Mod_{(j)}(W(k))$.
\end{lemma}
\begin{proof}
Obvious from the definition.
\end{proof}

\begin{corollary}\label{cor:Wittretract}
Let $k$ be a perfect graded field and $A=A_{(j)}$ a $p$-typical polar $k$-algebra with $d \nmid j$. Denote the regrading equivalence $\Mod_{(j)}(d) \to \Mod_{(1)}(p^{h(j,d)}-1)$ of Cor.~\ref{cor:primefields} by $\phi$. Then
\[
W_n(A) = \phi^{-1} \Bigl(W_n(\hull(\phi A))_{(1)}\Bigr).
\]
\end{corollary}
\begin{proof}
By the previous lemma, it suffices to consider the case $j=1$ and $d=p^h-1$.
The composite $A \to \hull(A) \to \hull(A)_{(1)}$ of $p$-polar algebras is an isomorphism by Cor.~\ref{cor:retract} and hence induces an isomorphism
\[
W_n(A) \to W_n(\hull(A)) \to W_n(\hull(A)_{(1)}) = W_n(\hull(A))_{(1)}.
\]
\end{proof}

If $A$ is a $p$-polar $k$-algebra over a perfect graded field $k$ (the possibility $k=k_0$ is included), we can generalize the ungraded construction of the group of  co-Witt vectors \cite{bauer:p-polar,fontaine:groupes-divisibles,bauer-carlson:tensorproduct} to the graded context as follows. 

\begin{defn}
Let $k=k_0[u^{\pm 1}]$ be a perfect graded field of characteristic $p$ and $A=A_{(j)}$ a $p$-typical polar $k$-algebra with $d = |u| \nmid j$.

Define the graded $W(k)$-module of co-Witt vectors $CW(A)$ by
\[
CW(A) = \phi^{-1}\Bigl(CW(\hull(\phi A))_{(1)}\Bigr),
\]
where $CW(\hull(\phi A))$ is defined as in Prop.~\ref{prop:gradedCW}.

For $A=A_{(j)}$ with $d \mid j$, we have that $A \cong A_0 \otimes_{k_0} k$ and define
\[
CW(A) = CW(A_0) \otimes_{W(k_0)} W(k),
\]
where the ungraded $CW(A_0)$ was defined in \cite{bauer:p-polar}.

For an arbitrary graded $p$-polar $k$-algebra $A=\prod_j A_{(j)}$, define
\[
CW(A) = \bigoplus_{j} CW(A_{(j)}).
\]

The submodule $CW^u(A)$ is defined in the same way as in Prop.~\ref{prop:gradedCW} and agrees, by Cor.~\ref{cor:Wittretract}, with the subset of $(\dots,a_1,a_0) \in CW(A)$ almost all of whose elements are zero.
\end{defn}
Note that the Frobenius and Verschiebung operations are well-defined on $CW(A)$ and $CW^u(A)$ for $A$ a $p$-polar $k$-algebra. 

This construction agrees with the one given in \cite{bauer:p-polar} when $k=k_0$ and $A=A_0$. If $k=k_0$ with arbitrary commutative graded $A$, note that for $j \neq 0$ almost all factors of
\[
\Bigl(\prod_{i \leq 0} A(i)\Bigr)_j = \prod_{i \leq 0} A_{jp^i}
\]
are zero (namely those where $jp^i \not\in \Z$), hence $CW(A)_j = CW^u(A)_j$ for $j \neq 0$ and $CW(A)_0 = CW(A_0)$, in other words, $CW(A) \cong CW^u(A) \oplus_{CW^u(A_0)} CW(A_0)$ as abelian groups. Having in mind the as yet unproven Theorem~\ref{thm:dieudonneoffree}, this corresponds to the fact that a connected, graded abelian Hopf algebra must be conilpotent.

\begin{example}
Let $k=k_0$ be perfect of characteristic $p$ and $A = k \langle x^{p^i} \mid i \geq 0\rangle$ the free $p$-polar algebra on a single generator $x$ in degree $2$. Then
\[
W_n(A)_{2p^i} = \{(a_0 x^{p^i},a_1 x^{p^{i+1}},\dots,a_n x^{p^{i+n}}) \mid a_i \in W(k)\} \cong W_n(k)
\]
and $W_n(A)_j=0$ if $j$ is not twice a power of $p$. The Verschiebung is given by
\[
V\colon W_n(A)_{2p^i} \to W_{n+1}(A(-1))_{2p^i}, \quad (a_0,\dots,a_n) \mapsto (0,a_0,\dots,a_n).
\]
Since $A_0=0$, we have that $CW(A) = CW^u(A) = \colim (W_0(A) \xrightarrow W_1(A(-1)) \xrightarrow \cdots)$ and thus
\[
CW(A)_{2p^i} = \{ (\cdots,a_{-1},a_0) \mid a_j \in A_{2p^{i+j}}\},
\]
which is understood to mean $a_j=0$ if $2p^{i+j} \not\in\Z$. Thus $CW(A)_{2p^i} = W_i(k)$, and the Frobenius and Verschiebung maps are given by
\[
V\colon CW(A)_{2p^i} \to CW(A)_{2p^{i-1}} \text{ the restriction map $W_i(k) \to W_{i-1}(k)$}
\]
and
\[
F\colon CW(A)_{2p^i} \to CW(A)_{2p^{i+1}}, \text{ the multiplication-by-$p$ map $W_i(k) \to W_{i+1}(k)$.}
\]
\end{example}

\section{Graded Dieudonné theory}

Throughout, let $k$ be a perfect graded field of characteristic $p$.

\begin{defn}
A \emph{graded Dieudonné module} over $k$ is a graded $W(k)$-module $M$ together with maps of $W(k)$-modules
\[
F\colon M \to M(1) \quad \text{and} \quad V\colon M(1) \to M
\]
satisfying $FV=p$ and $VF=p$.
We denote the category of Dieudonné modules (with the obvious definition of morphism) by $\DMod{k}$.

We call a Dieudonné module $M$ \emph{$p$-adic} if for every submodule $W(k)$-submodule $N < M$ of finite length, the submodule spanned by $\{V^n(N(n)) \mid n \geq 0\}$ is also of finite length. In other words, $M$ is $p$-adic if it is a colimit of finite-length $W(k)$-submodules $N$ closed under $V$ in the sense that $V(N(1)) \subseteq N$.
We call $M$ \emph{unipotent} if for any finite-length $W(k)$-submodule $N$ of $M$, $V^n(N) = 0$ for $n \gg 0$. We denote the full subcategories of $p$-adic and unipotent Dieudonné modules by $\DMod{k}^p$ and $\DModV{k}$, respectively.

Moreover, a Dieudonné module $M$ is called $F$-profinite if $M$ is profinite as a $W(k)$-module and has a fundamental system of neighborhoods consisting of $W(k)$-modules $N$ closed under $F$, i.e. such that $F(N) < N(1)$. The module $M$ is called \emph{connected} if the profinite completion of $F^{-1}M$ is trivial and \emph{étale} if $M \to F^{-1}M$ is an isomorphism. Denote the category of $F$-profinite Dieudonné modules by $\DDModF{k}$ and the full subcategories of connected (resp. étale) Dieudonné modules by $\DDModFc{k}$ (resp. $\DDModFet{k}$).
\end{defn}

Now consider the categories $\AbSch{k}^p$ resp. $\FGps{k}^p$. The category $\AbSch{k}^p$ is anti-equivalent to the full subcategory of bicommutative Hopf algebras $H$ over $k$ such that $H = \colim H[p^n]$, where $H[p^n]$ is the Hopf algebra kernel of the map $[p^n]\colon H \to H$. Similarly, the category $\FGps{k}^p$ is anti-equivalent to the full subcategory of the category of bicommutative complete Hopf algebras consisting of those $H$ such that $H=\lim \coker [p^n]$. The categories $\AbSch{k}$ and $\FGps{k}$ are anti-equivalent by Cartier duality, as are the categories $\AbSch{k}^p$ and $\FGps{k}^p$.

By \cite[Prop.~A.4]{bousfield:p-adic-lambda-rings}, any bicommutative Hopf algebra (and dually, every bicommutative complete Hopf algebra) over a field of characteristic $p>2$ splits naturally into an even part and an odd part:
\[
H = H_{e} \otimes H_{o},
\]
where $H_e$ is concentrated in even degrees and $H_o$ is an exterior algebra on primitive generators in odd degrees. The functor of primitives gives an equivalence between odd formal groups and oddly graded $k$-modules inverse to the exterior algebra functor. The odd part carries therefore very little information.

An odd Hopf algebra (or odd formal Hopf algebra) is automatically $p$-adic because $[p]$ is the trivial map on such Hopf algebras. 

This leads us to a simple proof of the odd part of Thm.~\ref{thm:freeschemefactorization}:
\begin{prop}\label{prop:oddfree}
Let $k$ be a graded field of characteristic $p>2$. For a graded commutative $k$-algebra $A$, denote by $\Fr_{o}(A)$ the odd part of $\Fr(A)$, i.e. $\Reg{\Fr_{o}(A)} = (\Reg{\Fr(A)})_{o}$. Then we have factorizations
\[
\begin{tikzcd}
\Alg_k^\op \arrow[r,"\Fr_{o}"] \ar[dr,swap,"(-)_{o}"] & \AbSch{k}^p\\
& (\Mod_k)_{o}^\op \ar[u,"\hat \Fr"],\\
\end{tikzcd}
\quad \text{and} \quad
\begin{tikzcd}
\alg_k^\op \arrow[r,"\Fr_{o}"] \ar[dr,swap,"(-)_{o}"] & \FGps{k}^p\\
& (\operatorname{mod}_k)_{o}^\op \ar[u,"\hat \Fr"],\\
\end{tikzcd}
\]
where the diagonal map assigns to $A$ the odd part of the underlying $k$-module $A$ and $\Reg{\tilde \Fr(M)} = \bigwedge(M)$.
\end{prop}
\begin{proof}
Let $H = \bigwedge(M)$ be an odd Hopf algebra, where $M$ is some oddly graded $k$-vector space, and $A \in \Alg_k$. Then
\begin{multline*}
\Hom_{\Mod_k}(M,A) = \Hom_{\Alg_k}(H,A) \\
= \Hom_{\Hopf{k}}(H,\Cof_{\odd}(A)) = \Hom_{\Mod_l}(M,P\Cof_{\odd}(A)).
\end{multline*}

Thus $P\Cof_{\odd}(A)\cong A$ and thus $\Cof_{\odd}(A) \cong \bigwedge(A)$.
\end{proof}

We will thus from now on focus only on evenly graded, commutative Hopf algebras (and affine and formal groups). The arbitrarily graded situation in characteristic $2$ is completely analogous.

\begin{prop} \label{prop:formalgroupsplitting}
Every formal group $\hat G$ splits naturally as $\hat G_0 \times \hat G_c$, where $\hat G_0$ is \'etale and $\hat G_c$ is connected. Dually, every bicommutative Hopf algebra $H$ splits as $H = H_m \otimes H_u$ where $H_u$ is unipotent (conilpotent) and $H_m$ is of multiplicative type.
\end{prop}
We denote the full subcategory of connected formal groups by $\FGpsc{k}$ and the corresponding full subcategory of unipotent Hopf algebras by $\Hopfu{k}$.
\begin{proof}
In the ungraded context, Fontaine proves this in \cite[\textsection I.7]{fontaine:groupes-divisibles}. We give an argument that works in the graded case for the reader's convenience, although no new ideas are needed.

Let $A$ be a finite-dimensional $k$-algebra. We say that $A$ is \'etale if it is a (finite) product of (finite) field extensions of $k$, where a field extension of a graded field $k$ is of course just an inclusion $k < k'$ of graded fields. Clearly, $A/\nil(A)$ is \'etale for any finite $k$-algebra $A$.

The formal group $G$ is called connected if $\Reg{G}$ is a pro-local $k$-algebra, and it is called \'etale if $\Reg{G}$ is pro-\'etale as a $k$-algebra.

Define $G_0$ by $G_0(A) = G(A/\nil(A))$ and let $G_c$ be the kernel of $G \to G_0$. Then $\Reg{G_c}$ is the pro-local ring having as maximal ideal the kernel of the counit $\epsilon\colon \Reg{G} \to k$.

We thus get a short exact sequence of formal groups
\begin{equation}\label{eq:connected-etale-splitting}
0 \to G_c \to G \to G_0 \to 0.
\end{equation}

If $A^e$ denotes the maximal \'etale subalgebra of $A$, we see that the map
\[
A^e\to A \to A/\nil(A)
\]
is an isomorphism and $\Reg{G_0} = \Reg{G}/\nil(\Reg{G})$, so $G_0$ is indeed \'etale, and the projection map $\Reg{G} \to \Reg{G_0}$ splits \eqref{eq:connected-etale-splitting}.
\end{proof}

For a graded perfect field $k=k_0[u^{\pm 1}]$, the separable (=algebraic) closure is given by $\bar k = \bar{k_0}[v^{\pm 1}]$, where $|v|=2$ and $u=v^{\frac d 2}$. (For $p=2$, $|v|=1$ and $u=v^d$. We will leave the necessary adjustments in this case to the reader.) We define the Galois group $\Gal(k' \mid k)$ of a field extension $k < k'$ to be the group of automorphisms of $k'$ fixing $k$. Since such a field extension is given by $k_0[u^{\pm 1}] < k'_0[v^{\pm 1}]$ with $u=v^e$, we find that
\[
\Gal(k'\mid k) = \Gal(k'_0 \mid k_0) \ltimes \mu_e(k'_0),
\]
where $\mu_e(k'_0) = \{\zeta \in k'_0 \mid \zeta^e=1\}$ acts by fixing $k'_0$ and mapping $v$ to $\zeta v$, and the Galois group $\Gal(k'_0 \mid k_0)$ acts in the natural way on $\mu_e(k'_0)$. We see that $k < k'$ is Galois iff $k_0 < k'_0$ is Galois and $k'_0$ contains a primitive $e$th root of unity. In particular, the profinite absolute Galois group is
\[
\Gamma = \Gal(k) = \Gal(\bar k \mid k) = \Gal(\bar k_0 \mid k_0) \ltimes \Z/\tfrac d 2 \Z. 
\]

To formulate Galois descent, we will consider $2$-functors
\[
\mathcal C \colon \Sep \to \Cat
\]
from the category of graded fields of characteristic $p$ and finite, separable field extensions to the $2$-category of categories. In particular, for any extension $k < k'$, we obtain an action of $\Gal(k'|k)$ on $\mathcal C(k')$. We say that $\mathcal C$ satisfies Galois descent, and call it a Galois descent category, if the natural map
\[
\mathcal C(k) \to \mathcal C(k')^{\Gal(k' \mid k)}
\]
is an equivalence of categories for every Galois extension $k \to k'$ in $\Sep$, where $\mathcal C(k')^{\Gal(k' | k)}$ denotes the $2$-categorical (homotopy) fixed points. This encapsulates the usual definition, as an object in $\mathcal C(k')^{\Gal(k' \mid k)}$ is an object with a $\Gal(k' \mid k)$-semilinear action.

\begin{lemma}[graded Galois descent] \label{lemma:galoisdescent}
The following functors are Galois descent categories, where in each case, $i\colon k \to k'$ is a morphism in $\Sep$:
\begin{enumerate}
	\item the functor $\Alg\colon k \mapsto \Alg_{k}$ with $\Alg(i)(A) = A \otimes_{k} k'$;
	\item the functor $\AbSch{}\colon k \mapsto \AbSch{k}^{\ev}$ with $\AbSch{}(i)(G) = G \times_{\Spec k} \Spec k'$;
	\item the functor $\FGps{}$ defined analogously;
	\item the functor $\DMod{}\colon k \mapsto \DMod{k}$ with $\DMod{}(i)(M) = M \otimes_{W(k)} W(k')$.
\end{enumerate}
Each of these categories has an ungraded analog, which we decorate with the letter $u$, e.g. $\Alg^u(k) = \{ \text{ungraded algebras over $k_0$}\}$. These are not Galois descent categories (over the category of graded fields $\Sep$), but the natural transformations
\begin{align*}
\Alg^u \to \Alg, & \quad A \mapsto A \otimes_{k_0} k\\
\AbSch{}^u \to \AbSch{}, & G \mapsto G \times_{\Spec k_0} \Spec k\\
\FGps{}^u \to \FGps{}, & G \mapsto G \times_{\Spec k_0} \Spec k\\
\DMod{}^u \to \DMod{}, & M \mapsto M \otimes_{W(k_0)} W(k),
\end{align*}
are equivalences on $k=k_0[u^{\pm 1}]$ with $|u|=2$.
\end{lemma}
\begin{proof}
For the first part, let $\Gamma = \Gamma_0 \ltimes C_e$ with $\Gamma_0 = \Gal(k'_0 \mid k_0)$ and $C_e = \mu_e(k_0')\cong \Z/e\Z$. Note that any separable extension factors as separable extensions $k < K < k'$, where $K=(k')^{C_e}$. Explicitly, if $k=k_0[u^{\pm e}]$ and $k'=k'_0[u^{\pm 1}]$ then $K = k'_0[u^{\pm e}]$. For the extension $k < K$, we have that
\[
A \mapsto A \otimes_k K = A \otimes_{k_0} K_0 = A \otimes_{k_0} k'_0
\]
induces an equivalence by ungraded Galois descent. For the extension $K < k'$, note that a $\Z/de\Z$-grading is the same as a $\Z/d\Z$-grading with an action of the group $C_e$, where the degree-$di$ parts can be recovered as the eigenspaces of the $C_e$-action.

For the second part, it suffices to show that $\C^u(k_0) \simeq \C_{k}$. Clearly, restriction to the degree-$0$ part induces an equivalence $\C^{\ev}_{k} \to \C^u_{k_0}$ with inverse $- \otimes_{k_0} k$, and this equivalence is compatible with the $\Gamma$-action.
\end{proof}

\begin{lemma}\label{lemma:etalefgpsaregammamodules}
Let $\Mod_{\Gal(k)}$ be the category of discrete abelian groups with a continuous action of the absolute Galois group $\Gal(k)$.
Then $\Mod_{\Gal}\colon k \mapsto \Mod_{\Gal(k)}$ is a Galois descent category equivalent to the category of \'etale formal groups.

Under this equivalence, formal $p$-groups correspond to discrete abelian $p$-groups with a continuous $\Gamma$-action.
\end{lemma}
\begin{proof}
The assignment $\Mod_{\Gal}$ becomes a functor by defining $i_*\colon \Mod_{\Gal(k)} \to \Mod_{\Gal(k')}$ to be the identity map, restricting the action of $\Gal(k)$ to that of $\Gal(k')$. This is clearly a Galois descent category since abelian groups with a $\Gal(k')$ and a $\Gal(k' \mid k)$-action are the same as abelian groups with a $\Gal(k)$-action.

As in the ungraded context, the equivalence is given by the functors (cf. \cite[\textsection I.7]{fontaine:groupes-divisibles})
\[
\begin{tikzcd}[column sep=5cm]
\{\text{\'etale formal groups}\} \arrow[r,bend left=5,"{G \mapsto \colim_{k \subseteq k' \subseteq \bar k} G(k')}"] & \Mod_{\Gamma}. \arrow[l,bend left=5,"{\Spf{\map^\Gamma(M,\bar k)} \mapsfrom M}"]
\end{tikzcd}
\]
\end{proof}

\begin{corollary}\label{cor:formalgpdecomp}
We have equivalences of Galois descent categories
\begin{align*}
\FGps{}^{\ev} &\simeq \Mod_{\Gal} \times (\FGpsc{})^{\ev}\\
\shortintertext{and}
\Hopf{}^{\ev} &\simeq \Mod_{\Gal} \times (\Hopfu{})^{\ev}
\end{align*}
\qed
\end{corollary}

\begin{thm} \label{thm:generaldieudonne}
There is an exact natural equivalence $\DieuFor$ between the following abelian Galois descent categories:
\begin{enumerate}
\item the category $(\FGps{}^p)^{\ev}$ of even formal $p$-groups; and
\item the category $(\DDModF{})^{\ev}$ of even $F$-profinite Dieudonné modules.
\end{enumerate}
This natural equivalence is represented by the formal group $CW_k$. 
\end{thm}
\begin{proof}
The statement unwinds to showing that the functor $\DieuFor$ represented by $CW_k$ from $\FGps{k}^p$ to $\DDModF{k}$ is an equivalence, and
that $\DieuFor$ is compatible with base change in the sense that for any finite, separable extension $k'$ of $k$,
\begin{equation} \label{eq:basechange}
\DieuFor_{k'} (G \otimes_k k') \cong \DieuFor_{k}(G) \otimes_{W(k)} W(k').
\end{equation}

This theorem is essentially well-known. An ungraded version appeared in \cite[\textsection III Théorème 1]{fontaine:groupes-divisibles}, cf. also \cite[Theorem~4.2]{bauer:p-polar}. 

The base change property \eqref{eq:basechange} was proved in the ungraded case in \cite[Prop.~III.2.2]{fontaine:groupes-divisibles}. For any Galois extension $k < k'$ with Galois group $\Gamma$ and finite $k$-algebra $A$,
\[
CW(A \otimes_k k')^\Gamma \cong CW(A)
\]
by Galois descent for finite algebras and because $\Gamma$ acts componentwise on $CW(A \otimes_k k') \subseteq \prod_{i \leq 0} A \otimes_k k'$.  If $G$ is a formal group over $k$ then
\[
\DieuFor_k(G) = \FGps{k}(G,CW_{k}) = \{ a \in CW_{k}(\Reg{G}) \mid \psi(a) = a \otimes 1 + 1 \otimes a \in CW_{k}(\Reg{G} \otimes_k \Reg{G})\}.
\]
and hence
\[
\DieuFor_{k'}(G \times_k k')^\Gamma = \{ a \in CW_{k'}(\Reg{G} \otimes_k k')^\Gamma \mid \psi(a) = a \otimes 1 + 1 \otimes a\}^\Gamma \cong \DieuFor_k(G).
\]
Equation \eqref{eq:basechange} follows from Galois descent for $W(k)$-modules (Lemma~\ref{lemma:galoisdescent}).

This shows that $\DieuFor$ is a natural transformation of Galois descent categories. To show it is an equivalence, we consider two cases: if $k=k_0$ is concentrated in degree $0$ then so is $W(k)$, and the classical theorem applies and preserves the grading. 

If on the other hand $k=k_0[u,u^{-1}]$ with $|u|=d$, we cannot use the ungraded result directly. By the first part of Lemma~\ref{lemma:galoisdescent}, it suffices to show that $\DieuFor$ is an equivalence after a finite Galois extension, so we may assume $d=2$, and by the second part of the same Lemma, the claim is once again reduced to the ungraded case.
\end{proof}

\begin{proof}[Proof of the formal case of Thm.~\ref{thm:dieudonneoffree}]
For an evenly graded, finite-dimensional $k$-algebra $A$ (arbitrarily graded if $p=2$), we have that
\[
\DieuFor(\Fr(A)) = \Hom_{\FGps{k}}(\Fr(A),CW_k) = \Hom_{\FSch{k}}(\Spec A, CW_k) = CW(A). 
\]
\end{proof}

\begin{proof}[Proof of the formal case of Thm.~\ref{thm:freeschemefactorization}]
By Prop.~\ref{prop:oddfree}, it suffices to consider the case where $p=2$ or $\Fr_{\ev}$, the even part of the free formal group functor. 

The splitting of Prop.~\ref{prop:formalgroupsplitting} gives a splitting of the functor $\Fr_{\ev}$ as
\[
\Fr_{\ev} = \Fr_{\ev,0} \times \Fr_{\ev,c}.
\]
The Dieudonné functor restricts to an equivalence between connected formal groups and connected Dieudonné modules, so Thm.~\ref{thm:dieudonneoffree} gives an  extension $\hat\Fr_{\ev,c}$ by
\[
\tilde\Fr_{\ev,c}(A) = \left((\DieuFor)^{-1}CW(A)\right)_c.
\]
The same argument would work for étale, $p$-adic formal groups, but Theorem~\ref{thm:freeschemefactorization} does not require $p$-adicness. Instead, we observe as in \cite[Lemmas 4.3, 4.5]{bauer:p-polar} that $\colim_{k \subseteq k' \subseteq \bar k} \Fr_{\ev,0}(A)(k') = \Z\langle \Hom_{\Alg_k}(R,\bar k)\rangle$ is well-defined for $p$-polar algebras. By Lemma~\ref{lemma:etalefgpsaregammamodules}, this gives $\tilde \Fr_{\ev,0}$.
\end{proof}

We will now turn to the affine case. The analog of Thm.~\ref{thm:generaldieudonne} is the following:
\begin{thm}\label{thm:affinedieudonne}
Let $k$ be a perfect graded field of characteristic $p$. Then there is an equivalence $D$ between the following Galois descent categories:
\begin{enumerate}
	\item The category $(\AbSch{}^p)^{\ev}$ of $p$-adic, even group schemes; and
	\item the category $(\DMod{}^p)^{\ev}$ of $p$-adic, even Dieudonné modules.
\end{enumerate}
\end{thm}
\begin{proof}
The equivalence is given by the composition
\[
\AbSch{k}^{p,ev} \xrightarrow{d} \FGps{k}^p \xrightarrow{\DieuFor} (\DDModF{k})^{\ev} \xrightarrow{I} \DMod{k}^{p,\ev},
\]
where $d$ denotes Cartier duality and $I$ denotes Matlis (Poincaré) duality
\[
M \mapsto \Hom^c_{W(k)}(M,CW(k)),
\]
the group of continuous homomorphisms into $CW(k)$. Since both $d$ and $I$ are anti-equivalences and because of Thm.~\ref{thm:generaldieudonne}, $D$ is an equivalence. To see this is an equivalence of Galois descent categories, we need to see that both $d$ and $I$ are. Cartier duality is given by taking linear duals on the level of representing objects, and thus
\[
\Hom_{k'}(H \otimes_k k',k') = \Hom_{k}(H,k') \cong \Hom_k(H,k) \otimes_k k'
\]
for finite extension $k \to k'$. For Matlis duality, we see that
\begin{multline*}
\Hom^c_{W(k')}(M \otimes_{W(k} W(k'),CW(k')) = \Hom^c_{W(k)}(M,CW(k) \otimes_{W(k)} W(k'))\\ \cong \Hom^c_{W(k)}(M,CW(k)) \otimes_{W(k)} W(k')
\end{multline*}
since $W(k')$ has finite length as a $W(k)$-module.
\end{proof}

We next study how the free $p$-adic affine abelian group functors behave with respect to field extensions and Galois descent.

\begin{lemma} \label{lemma:freeandbasechange}
The natural transformation $\Fr\colon \Alg \to \AbSch{}$ is a natural transformation of Galois descent categories.
\end{lemma}
\begin{proof}
We need to prove that for any finite separable extension $k \to k'$ of graded fields, $\Fr(A \otimes_k k') \cong \Fr(A) \times_{\Spec k} \Spec{k'}$ for any graded $k$-algebra $A$.
Let $\Fr_u$ be the unipotent part and $\Fr_m$ be the part of multiplicative type, corresponding to the connected and the étale parts, respectively, of the Cartier dual formal group.

The unipotent part $\Fr_u(A)$ is represented by the cofree, conilpotent, cocommutative Hopf algebra on $A$,
\[
\Cof^u(A) = \bigoplus_{n \geq 0} (A^{\otimes_k n})^{\Sigma_n},
\]
and since $k'$ is flat over $k$, taking $\Sigma_n$-fixed points commutes with base change. 

The multiplicative part $\Fr_m(A)$ is represented by $\bar k[(A \otimes_k \bar k)^\times]^\Gamma$. The argument is the same as in \cite[proof of Thm.~1.3]{bauer-carlson:tensorproduct}, using graded Galois descent:

Firstly, if $k=\bar k$ and $H=k[M]$ is a Hopf algebra of multiplicative type, then
\[
\Hom_{\Hopf{k}}(H,k[A^\times]) \cong \Hom(M,A^\times) \cong \Hom_{\Alg_k}(k[M],A),
\]
so that the claim holds when $k$ is algebraically closed. Now let $k$ be arbitrary perfect, write $\bar A = A \otimes_k \bar k$, and let $H$ be a Hopf algebra of multiplicative type with $H \otimes_k \bar k \cong \bar k[M]$. Then
\[
\Hom_{\Hopf{k}}(H,\bar k[\bar A^\times]^\Gamma) = \Hom_{\Hopf{\bar k,\Gamma}}(\bar k[M], k[\bar A^\times])
\]
by Lemma~\ref{lemma:galoisdescent}, and the latter group is isomorphic to 
\[
\Hom^\Gamma(M,\bar A^\times) \cong \Hom_{\Alg_{\bar k,\Gamma}}(\bar k[M],\bar A) \cong \Hom_{\Alg_k}(H,A),
\]
again using Lemma~\ref{lemma:galoisdescent}, this time for algebras.

Now, if $k'$ is a Galois extension of $k$, we have
\[
\bar k[\bar A^\times]^{\Gal(k)} \otimes_k k' = \Bigl(\bar k[\bar A^\times]^{\Gal(k')}\Bigr)^{\Gal(k' \mid k)} \otimes_k k' \cong \bar k[\bar A^\times]^{\Gal(k')},
\]
using Lemma~\ref{lemma:galoisdescent} once more.
\end{proof}

\begin{corollary}
The free $p$-adic affine group functor $\Fr^p\colon \Alg \to \AbSch{}^p$ induces a natural transformation of Galois descent categories.
\end{corollary}
\begin{proof}
The functor $\Fr^p$ is just $\Fr$ followed by $p$-completion. By Cor.~\ref{cor:formalgpdecomp} and since unipotent groups are automatically $p$-complete, it suffices to show that $p$-completion commutes with separable base change in the opposite category of $\Mod_{\Gal}$. Since $p$-completion is dual to taking the subgroup of $p$-power torsion elements in abelian groups, this claim boils down to the obvious statement that restricting group actions and taking $p$-power torsion elements commutes in abelian groups.
\end{proof}

\begin{proof}[Proof of the affine part of Theorem~\ref{thm:dieudonneoffree}]
Since for the ungraded Dieudonné functor $D^u$, of the same form as in the statement, this was proved in \cite{bauer:p-polar}, we will proceed by showing that both the left hand side $L(A) = D(\Fr(A))$ and the right hand side $R(A) = CW^u(A) \oplus (\mu_{p^\infty}(A \otimes_k \bar k) \otimes W(\bar k))^{\Gal(k)}$ are natural transformations between the Galois descent categories $\Alg$ and $\DMod{}$. 

For $L$, this is guaranteed by Lemma~\ref{lemma:freeandbasechange} and Thm.~\ref{thm:affinedieudonne}. For $R$, it is true for $CW^u$ as a subfunctor of $CW$ since $CW^u(A \otimes_k k')^\Gamma \cong CW^u(A)$ as in the proof of Thm.~\ref{thm:generaldieudonne}. For the second factor,
\[
R_2(A) = \Bigl(\mu_{p^\infty}(R \otimes_k \bar k) \otimes W(\bar k)\Bigr)^{\Gal(k)},
\]
it is almost tautological. Indeed, if $k \to k'$ is a Galois extension with group $\Gamma$, without loss of generality assumed to be a subfield of $\bar k$, then
\[
R_2(A \otimes_k k')^\Gamma = \Biggl(\Bigl(\mu_{p^\infty}(R \otimes_k k' \otimes_{k'} \bar k) \otimes W(\bar k)\Bigr)^{\Gal(k')}\Biggr)^\Gamma = R_2(A)
\]
and hence $R_2(A \otimes_k k') \cong R_2(A) \otimes_{W(k)} W(k')$ by Galois descent for $W(k)$-modules.

Now if $k=k_0$ is an ungraded field, Theorem~\ref{thm:dieudonneoffree} follows from \cite{bauer:p-polar} directly since the given isomorphism constructed there respects any gradings. If, on the other hand, $k=k_0[u^{\pm 1}]$ then we can assume, by the above descent argument, that $|u|=2$. 

The result then follows by observing that the diagram
\[
\begin{tikzcd}
\AbSch{k_0[u^{\pm 1}]}^{p,\ev} \arrow[r,"D_{k_0[u^{\pm 1}]}"] & \DMod{k_0[u^{\pm 1}]}\\
\AbSch{k_0}^{p,u} \arrow[u,"\simeq"] \arrow[r,"D^u_{k_0}"] & \DMod{k_0}^u \arrow[u,"\simeq"]
\end{tikzcd}
\]
commutes. 
\end{proof}

\begin{remark}
The reader might wonder if Theorem~\ref{thm:freeschemefactorization} cannot be directly derived from the ungraded case using Galois descent technology. The problem is that $\Pol_p$ is not a Galois descent category.
\end{remark}

\section{Properties and applications}

The factorization of the free formal group functor (Thm.~\ref{thm:freeschemefactorization}) induces a factorization
\[
\begin{tikzcd}
(\Pro-\alg_k)^{\op} \arrow[r,"\Fr"] \ar[dr,swap,"\pol"] & \FGps{k}\\
& (\Pro-\pol_p(k))^{\op}, \ar[u,"\tilde \Fr"]\\
\end{tikzcd}
\]
where the functors denoted by $\Fr$ and $\tilde \Fr$ are the unique extension of the functors from Thm.~\ref{thm:freeschemefactorization} that commute with directed colimits. In this section, we will concentrate on the free unipotent, resp. connected, construction only.

\begin{lemma}\label{lemma:formalfreeadjoint}
The functor $\hat \Fr^c\colon (\Pro-\pol_p(k))^{\op} \to \FGps{k}$ commutes with all colimits and has a right adjoint $V$.
\end{lemma}
\begin{proof}
For the odd part, the functor $\hat\Fr_o$ factors as
\[
\hat\Fr_o\colon (\Pro-\pol_p(k))^{\op} \xrightarrow{U} (\Pro-\operatorname{mod}_k)_o^{\op} \xrightarrow{\hat{\Fr}} \AbSch{k}^p
\]
(cf. Prop.~\ref{prop:oddfree}), where $U$ is the forgetful functor. Then an adjoint is given by the composition of the adjoint of $\hat\Fr$ (which is the functor of primitives) and an adjoint of $U$. The latter is the objectwise free $p$-polar algebra functor, which works because the free $p$-polar algebra on an odd finite-dimensional $k$-module is again finite dimensional (it is a sub-$p$-polar algebra of the exterior algebra).

So we can restrict our attention to even formal groups. For the free connected formal group $\Fr_c$, it suffices to show that its composition with the Dieudonné equivalence $\DieuFor$ commutes with the stated colimits, and by Theorem~\ref{thm:dieudonneoffree}, it is therefore enough to show that $CW_k^c\colon \Pro-\pol_p(k) \to \DMod{k}$ commutes with all limits. But $CW_k^c$ is a formal group, i.e. representable. 

The existence of an adjoint follows from Freyd's special adjoint functor theorem once we show that $\Pro-\pol_p(k)$ is complete, well-powered, and possesses a cogenerating set. Any pro-category of a finitely complete category, such as $\pol_p(k)$, is complete, and any pro-category has constant objects as a cogenerating class. Since $\pol_p(k)$ has a small skeleton, the condition on a cogenerating set is satisfied.
To see that $\Pro-\pol_p(k)$ is well-powered, observe that a subobject $S < A$ for $A \in \Pro-\pol_p(k)$ is in particular a sub-pro-vector space. By \cite[Prop. 4.6]{artin-mazur:etale}, a monomorphism in $\Pro-\Mod_k$ can be represented by a levelwise monomorphism. Thus if $A\colon I \to \Mod_k$ represents a pro-finite $k$-module with $\#\operatorname{Sub}(A(i)) = \alpha_i$ for some (finite) cardinals $\alpha_i$ then $\#\operatorname{Sub}(A) \leq \prod_{i \in I} \alpha_i$; in particular, it is a set.
\end{proof}

\begin{lemma}\label{lemma:affinefreeadjoint}
The functor $\hat \Fr^u\colon \Pol_p(k)^{\op} \to \AbSch{k}^u$ commutes with all colimits and filtered limits, and has a right adjoint $V$.
\end{lemma}
\begin{proof}
As in Lemma~\ref{lemma:formalfreeadjoint}, the right adjoint on odd affine groups is given by the functor of primitives follow by the free $p$-polar algebra functor, so we will restrict our attention to even $p$-adic affine groups. By Theorem~\ref{thm:dieudonneoffree}, it suffices to show that the functors $CW^u\colon \Pol_p(k) \to \DModV{k}$ commutes with all limits, which looks wrong until one realizes that limits in $\DModV{k}$ are not the same as limits in $\DMod{k}$.

The functor $CW^u(A)$ commutes with finite limits (it is ind-representable), so it suffices to show it commutes with infinite products. Indeed, the natural map
\[
CW^u(\prod_i A_i) \to \prod_i CW^u(A_i)
\]
is an isomorphism; an element in the right hand side is a set of elements $(x_i \in CW^u(A_i))$ such that there is an $n \gg 0$ such that $V^n(x_i)=0$ for all $i$. 



The commutation with filtered colimits is straightforward: $CW^u$ commutes with them because it is a colimits of functors represented by small objects (polarizations of finitely presented algebras).

For the existence of an adjoint, we apply again the special adjoint functor theorem in the form of \cite[Thm.~1.66]{adamek-rosicky:presentable-accessible}. The category $\Pol_p(k)$ is locally presentable and the functor $D \circ \hat Fr^u$, as just shown, is accessible (commutes with $\omega$-filtered colimits) and commutes with all limits.
\end{proof}

Note that this implies, by taking adjoint functors in Thm.~\ref{thm:freeschemefactorization}, that the algebra underlying a unipotent Hopf algebra $H$ is always of the form $\hull(V(H))$, i.e. free over a $p$-polar $k$-algebra, and the pro-finite algebra underlying a complete connected Hopf algebra $H$ is always of the form $\hull(V(H))$, i.e. free over a profinite $p$-polar $k$-algebra. Of course, this is also a direct corollary of Borel's work on the structure of algebras underlying Hopf algebras \cite{borel:homologie-des-groupes-de-Lie,milnor-moore:hopf} .

\begin{lemma}\label{lemma:cofreeiscofreeascoalg}
Let $A$ be a $p$-polar $k$-algebra and $H=\Cof^u(A)$ the unipotent Hopf algebra representing the $p$-adic affine group $\hat \Fr^u(A)$. Then $H$ is isomorphic, as a pointed coalgebra, to the symmetric tensor coalgebra on the $k$-vector space $A$.
\end{lemma}
\begin{proof}
The symmetric tensor coalgebra on a vector space $V$ is given by
\[
S(V) = \bigoplus_{i \geq 0} S^i(V) \quad \text{with} \quad S^i(V)= \bigl(V^{\otimes i}\bigr)^{\Sigma_i},
\]
and it is a pointed coalgebra by the inclusion $k \cong S^0(V) \subset S(V)$. A pointed coalgebra $C$ is conilpotent if for each $x \in C$, $\psi^N(x) \in C^{\otimes (N+1)}$ maps to $0$ in $\bar C^{\otimes (N+1)}$, where $\bar C$ is the cokernel of the pointing. The coalgebra $S(V)$ is conilpotent and, indeed, the right adjoint to the forgetful functor $U$ from the category $\Coalg_k^u$ of pointed, conilpotent, cocommutative coalgebras to $k$-vector spaces, mapping a coalgebra $C$ to $\bar C$.

Note that a Hopf algebra is unipotent if and only if its underlying pointed coalgebra is conilpotent. The claim is that the diagram
\[
\begin{tikzcd}
\Pol_p(k) \arrow[r,"\Cof^u"] \arrow[d,"U_1"] & \Hopf{k}^u \arrow[d,"U_2"]\\
\Mod_k \arrow[r,"S"] & \Coalg_{k}^u
\end{tikzcd}
\]
$2$-commutes.  By taking left adjoint functors, this is equivalent to the $2$-commutativity of the square in the diagram
\[
\begin{tikzcd}
\Pol_p(k) & \Hopf{k}^u \arrow[l,swap,"V"] \\
\Mod_k  \ar[u,"\Fr"] & \Coalg_{k}^u. \arrow[l,"U"] \arrow[u,"\Fr"]
\end{tikzcd}
\]
The free commutative Hopf algebra on $C \in \Coalg_k^u$ is given by the symmetric algebra $\Sym(\bar C)$, and hence there is a natural map of $k$-modules $C \to \Sym(\bar C)$ given in degree $0$ by the augmentation and in degree one by the projection $C \to \bar C$. This map $\phi\colon U(C) \to U_1(V(C))$ is adjoint to a map $\phi\colon \Fr(U(C)) \to V(\Fr(C))$ in $\Pol_p(k)$. To see that this map is an isomorphism, we consider its image under the conservative functor
\[
\hull\colon \Pol_p(k) \to \Alg_k.
\]
Since $\hull\circ V\colon \Hopf{k}^u\to \Alg_k$ is the forgetful functor and $\hull \circ \Fr\colon \Mod_k \to \Alg_k$ is the symmetric algebra functor, we see that $\hull(\phi)$ is the identity on $\Sym(C)$.
\end{proof}

\begin{corollary}
Let $H$ be a unipotent Hopf algebra which is unipotent cofree on a $p$-polar $k$-algebra $A$. Then $A$ is isomorphic to the vector space of primitive elements $PH$.
\end{corollary}
\begin{proof}
If $H$ is as in the statement then $U_2(H) \cong S(PH)$ as coalgebras, but by the preceding Lemma, $U_2(H) \cong S(U_1(A))$. Applying the functor $P$ and noting that $P(S(M)) \cong M$, we find that $PH \cong U_1(A)$.
\end{proof}

\begin{remark}
It is not true that $V(H)=P(H)$ in general, or that $P(H)$ is a $p$-polar algebra. Also, if $H$ is a unipotent Hopf algebra whose underlying pointed unipotent coalgebra is cofree, $H$ is not necessarily cofree over a $p$-polar $k$-algebra. For example, consider the graded Hopf algebra $H$ dual to $H^*=k[x,y]$ with $|x|=j>0$, $|y|=p^2j$, $x$ primitive and $\psi(y) = y \otimes 1 + 1 \otimes y + \sum_{i=1}^{p-1} \frac1{i!(p-i)!} x^{pi} \otimes x^{p(p-i)}$. Then the primitives $PH$ are dual to the indecomposables $Q(H^*) = \langle x,y\rangle$, i.e. $PH = \langle a,b\rangle$ with $|a|=j$, $|b|=p^2j$. Suppose $H$ was cofree, to that $PH$ is a $p$-polar $k$-algebra by the corollary above. For degree reasons, $PH$ cannot carry any but the trivial $p$-polar algebra structure, $PH \cong \langle a \rangle \times \langle b \rangle$. As a right adjoint, $\Cof^u$ commutes with products and hence $\Cof^u(PH) = \Cof(\langle a \rangle) \otimes \Cof(\langle b \rangle) = (k[x] \otimes k[y])^* = H'$. But $H \not\cong H'$ as Hopf algebras since $P(H^*) = \langle x,x^p,x^{p^2},\dots \rangle$, while $P((H')^*) = \langle x,x^p,\dots,y,y^p,\dots\rangle$, a contradiction.
\end{remark}
\begin{proof}[Proof of Thm.~\ref{thm:lambdapcofree}]
Let $\Lambda_p = k[\theta_{j,0},\theta_{j,1},\dots]$ be the Hopf algebra representing the functor of $p$-typical Witt vectors. Denote by $\eta\colon \Lambda_p \to \Cof^u(V(\Lambda_p))$ the unit of the adjunction. Applying the functor of primitives, since $\Cof^u(V(\Lambda_p))$ is cofree as a coalgebra by Lemma~\ref{lemma:cofreeiscofreeascoalg}, we obtain a map
\[
P\eta\colon P(\Lambda_p) \to V(\Lambda_p).
\]
Now $V(\Lambda_p) = \pol_{(j)}(\Lambda_p) = k\langle \theta_{j,i}^{p^k} \mid i,k \geq 0\rangle$ and $P(\Lambda_p) = k\langle \theta_{j,0}^{p^k} \mid k \geq 0\rangle$, and $P\eta$ is the inclusion map. We see that $P(\Lambda_p)$ is in fact a direct factor of $V(\Lambda_p)$ as a $p$-polar algebra, with an retraction $p\colon V(\Lambda_p) \to P(\Lambda_p)$ given by
\[
p(\theta_{j,i}) \mapsto \begin{cases} \theta_{j,0}; & i=0\\
0; & \text{otherwise} \end{cases}
\]
Taking adjoints, we obtain a map of Hopf algebras $q\colon \Lambda_p \to \Cof^u(P\Lambda_p)$. A map of unipotent Hopf algebras is injective iff it is injective on primitives, and $Pq\colon P\Lambda_p \to P(\Cof^i(P\Lambda_p))\cong P\Lambda_p$ is the identity, so $q$ is injective. By dimension considerations, it must also be surjective.
\end{proof}

\bibliographystyle{alpha}
\bibliography{bibliography}

\end{document}